\documentclass{amsart}

\usepackage{url}
\usepackage{amscd}
\usepackage{amsfonts}
\usepackage{amssymb}
\usepackage{euscript}
\usepackage{amsmath}
\usepackage{amsthm}
\usepackage[matrix, arrow, curve]{xy}
\usepackage{mathrsfs}

\usepackage{tikz}
\usepackage{tikz-cd}
\usetikzlibrary{arrows.meta}
\usepackage{dsfont}

\newcommand{\Hol}[1]{\mathop{Hol_{#1}}}

\newcommand{\M}[1]{\mathcal{M}_{0, #1}}
\newcommand{\Ms}[1]{\overline{\mathcal{M}}_{0, #1}}
\newcommand{\set}[1]{\underline{\mathbf{#1}}}
\newcommand{\PP}{\mathbb{P}^1}
\newcommand{\Gm}{\mathbb{G}_m}
\newcommand{\Gms}{\Gm\setminus\{1\}}
\newcommand{\EE}{\mathcal{E}}
\newcommand{\FF}{\mathcal{F}}
\newcommand{\LL}{\mathcal{L}}
\newcommand{\con}[2]{#1 \ast^! #2}
\newcommand{\jj}{\mathop{j_{\infty!}j_{0*}}}

\newcommand{\nea}[2]{\mathop{\psi_{#1}(#2)}}
\newcommand{\van}[2]{\mathop{\phi_{#1}(#2)}}
\newcommand{\Nea}[2]{\mathop{\Psi_{#1}(#2)}}
\newcommand{\Van}[2]{\mathop{\Phi_{#1}(#2)}}
\newcommand{\Shs}{\mathop{\mathop{\mathrm{Perv}}}(\Gm , 1)}

\newcommand{\Ho}[1]{\sideset{}{^{#1}}{\operatorname{\mathnormal{H}}}}
\newcommand{\Rde}{\operatorname{\mathbf{R}}}
\newcommand{\var}{var}

\newcommand{\ph}[1]{\varphi_{#1}}

\newcommand{\R}{\mathop{{\mathbf{R}}}}
\newcommand{\Rm}{\mathop{{\tau^{\le 0}\operatorname{\mathbf{R}}}}}

\theoremstyle{plain}
\newtheorem{prop}{Proposition}
\newtheorem{theorem}{Theorem}

\newtheorem*{theorem*}{Theorem}

\theoremstyle{definition}

\theoremstyle{remark}
\newtheorem{rem}{Remark}

\begin{document}

\title[Convolution]{Multiplicative convolution and double shuffle relations: convolution}
\author{Nikita Markarian}

\begin{abstract}
This is the first of two parts of a project devoted to 
a geometric interpretation of the Deligne--Terasoma approach
to regularized double shuffle relations. The central fact of this approach
is the isomorphism between vanishing cycles of multiplicative convolution of 
certain perverse sheaves and the tensor product of vanishing cycles,
which may be written in two different ways.
 These isomorphisms depend on a choice of a functorial
isomorphism $\varphi$ between vanishing cycles of a perverse sheaf on $\mathbb{C}^*$ and 
 cohomology of its certain extension on $\mathbb{P}^1$. The isomorphism chosen in the present paper
guarantees compatibilities with the  isomorphisms. 

In the second part of the project,
we will study other choices of $\varphi$. We will see that its compatibilities with
convolution imply regularized double shuffle relations. In particular,
associator relations imply them.
\end{abstract}

\email{nikita.markarian@gmail.com}

\date{}

\address{UMR 7501, Université de Strasbourg,
7 rue René-Descartes,
67084 Strasbourg Cedex, France}
\maketitle
%
%\epigraph{Художник Орлов пытался одним махом  объединить  чай,  бамбук  и
%лодку  на  вывеске  прямой  линией. Он пошел кратчайшим путем к
%истине и промахнулся.}{Юрий Коваль. Самая легкая лодка в мире}

\section*{Introduction}

Multiple zeta values are periods of iterated integrals 
and thereby are coefficients of the Drinfeld associator.
As such, they obey the set of relation implied by
the relations the associator satisfies.
Another set of relations, regularized double shuffle relations,  was introduced in \cite{IKZ}, and nicely formalized in \cite{Racinet}.  

The unfinished project \cite{DLet, DTerICM, DTerPr} of
Deligne and Terasoma, developed later in \cite{Enriquez2021,Enriquez2022,Enriquez2023},
connects double shuffle relations with 
convolution of perverse sheaves. The main aim of this project
was to show that associator relations imply regularized double shuffle
relation. Although since then this fact was proved in different ways, see \cite{Furusho2011, HiroseSato, Furusho2022},
the geometric approach is still of great interest and has applications we will discuss in the second part of the paper.

The present project grew up
from attempts to interpret the Deligne--Terasoma approach.
Note that although our approaches have the same main characters,
perverse sheaves and convolution, there are some differences. 
Firstly, the central role in the Deligne--Terasoma approach
is played by the quotient category as in Theorem \ref{Quotient} below,
we do not use this category at all. Another subject, playing an important
role there, Hodge structure and ``fake Hodge structure'' also
do not appear in our approach. It would be very interesting
to fill these gaps.

In the first part of the paper presented below, we discuss only 
interaction between
multiplicative convolution of perverse sheaves and vanishing cycles.
A paradigmatic example of this relation is the 
Thom--Sebastiani theorem \cite{Sebastiani1971, Massey2001},
which states that vanishing cycles of the direct sum of singularities
is equal to tensor product of the ones of summands.
The initial proof (\cite{Sebastiani1971}) is transcendental in nature, essentially they
show that the Milnor fiber of the sum is homeomorphic to the join
of Milnor fibers of the summands.

The vanishing cycles of singularities are vanishing cycles of the pushforward
of the constant sheaf and the direct sum of singularities corresponds essentially to the 
additive convolution. In contrast, we, following \cite{DLet}, consider the category
of perverse sheaves on $\Gm$ smooth outside $1$
and their multiplicative convolution. 

On this category two functors are defined:  vanishing cycles at $1$ and the $0$th cohomology of the extension $\Ho{0}(\PP,\, \jj -)$. They are isomorphic, we define
an isomorphism $\phi_I$ between them. This isomorphism
is given by shrinking on the interval $[0,1]$ and is transcendental in nature.
It means that it does not respect arithmetic structures such as Hodge structure.
This fact could be a link between our approach and the Deligne--Terasoma one.

Both functors are tensor with respect to the multiplicative convolution.

For vanishing cycles, the isomorphism between vanishing cycles of convolution
and the product of vanishing cycles of factors is given by the local calculation
of the additive convolution of the Verdier specializations of perverse sheaves.
Note that
Verdier specialization is another possible link between Deligne--Terasoma and  our approaches. Indeed, de Rham and Betti fiber functors
in \cite{DGal}
are defined by means of the Verdier specializations of a vector bundle
at infinite points.

The isomorphism between $\Ho{0}(\PP,\, \jj -)$ of convolution 
and tensor product of $\Ho{0}(\PP,\, \jj -)$ of factors is essentially
the K\"unneth formula. Alternatively, one may present this isomorphism
in terms of nearby cycles at infinity, this construction 
reminds the join in the proof of the  Thom--Sebastiani theorem 
mentioned above.

Isomorphism $\phi_I$ connects the two tensor products above.
The main result of the text, Theorem \ref{Theorem}, states that this isomorphism
is compatible with isomorphisms given by $\phi_I$ of factors.
This follows from the existence of the pentagon in $\Ms{5}$ 
restricted by intervals, which are copies of the interval $I$,
used to define  isomorphism $\phi_I$.

In the second part of the text, we will replace the straight line $I$ with a pro-unipotent path.
It is not surprising that the pentagon equation on this path
implies analogous compatibilities  of the corresponding isomorphism
of functor with the convolution product. This compatibility
may be reformulated as regularized shuffle relations between
the coefficients of the element of the pro-unipotent
completion of the  fundamental group of $\M{4}$ corresponding to the path.

%I would recommend to geometrically oriented reader to start with the Section \ref{poly} and keep in mind the picture from the end of the text.
%
%
%
%
%There are two isomorphisms between $\Van{}{\EE}\otimes \Van{}{\FF}$
%and $\Van{}{\EE}\otimes \Van{}{\FF}$: ``local'' and ``global''.

%We recommend to reader noticing to themselves each time we use this isomorphism.
%In the next part we will replace this isomorphism with a more general
%one. We will se, that regularized double shuffle relations
%are  in sense equivalent to compatibility conditions of this isomorphism
%with convolution. It is not surprising that they are implied by the
%pentagon relations.

\medskip
{\bf Acknowledgments.} I am grateful to B.~Enriquez, M.~Finkelberg H.~Furusho and M.~Kapranov for fruitful and deep discussions and  support. 
I would like to thank the Max Planck Institute for Mathematics and IHES for hospitality and perfect
        work conditions. The project is 
supported by funding of the program PAUSE and of the ITI IRMIA++.

\subsection*{Notations}
Everything is over $\mathbb{C}$, althoug some statement are true in a more general situation.

$\PP$ is the projective line, $\Gm=\PP\setminus\{0, \infty\}$ is the multiplicative group.

$\Shs$ is the category of perverse sheaves on $\Gm$ smooth outside 1.

We omit annotation of derived functors: $j_*$ means $\Rde j_*$,
$i^!$ means    $\Rde i^!$ and so on.

For a map $f\colon X \to  D \subset \mathbb{C}$
 from a  complex variety X to a disc
$D \subset \mathbb{C}$ with the central fiber  $X_0 = f^{-1}(0)$, the inclusion map $i\colon X_0 \to X$, and a complex $\FF$ of constructible
sheaves $\FF^\bullet$ on $X$, denote by $\phi \FF^\bullet$ and 
$\psi \FF^\bullet$ the complex of vanishing and nearby cycles on $X_0$.
Functors $\van{f}{-}[-1]$ and $\nea{f}{-}[-1]$ conserve the category of perverse sheaves.
They form  
Milnor triangles:
\begin{equation}
\begin{tikzcd}[cramped]
\nea{f}{\FF^\bullet}[-1]  \arrow[r, "can"]
& \van{f}{\FF^\bullet}[-1]  \arrow[r] 
& i^*\FF^\bullet \arrow[r, "+1"] 
& {}
\end{tikzcd}
\label{can}
\end{equation}

\begin{equation}
\begin{tikzcd}[cramped]
i^! \FF^\bullet \arrow[r]
& \van{f}{\FF^\bullet}[-1]  \arrow[r,"var"]
& \nea{f}{\FF^\bullet}[-1] \arrow[r, "+1"] 
& {}
\end{tikzcd}
\label{var}
\end{equation}

We denote  cohomologies of vanishing and nearby cycles  at a point   by capital letters:
\begin{equation}
\Van{f}{\FF}=\Ho{0}(\van{f}{\FF^\bullet}[-1]) \quad \mbox{and} \quad \Nea{f}{\FF}=\Ho{0}(\nea{f}{\FF^\bullet}[-1] )
\end{equation}

%REPLACE!
%For points of $\PP$ to emphasize the point and the function we use subscript, for example,
%$\Van{\ul}{-}$ denotes vanishing cycles at $1$ along function $1-z$,
%$\Van{\zr}{-}$ denotes vanishing cycles at $0$ along function $z$ and so on.
%
%
%
%We denote by $\ul$ and $\zr$ the corresponding infinite points.

Denote by $i_a\colon pt \to \PP$ the embedding of point $a$ in $\PP$
for $a= 0,\infty, 1$.

\section{Deligne fiber functor}
\label{deligne}

\subsection{Vanishing cycles}
\label{van1}

Firstly, prove a general useful result.
Let $I$ be the real interval $[0, 1]$ in $\PP$ not containing $\infty$, let
$i_I\colon I\to \PP$ be the closed embedding.
Denote by $j_\infty$ is the inclusion of $\mathbb{A}^1$ into $\PP$.
 
\begin{prop}[Shrinking]
For a complex of sheaves $\FF^\bullet$ on  $\mathbb{A}^1$ constructible with respect to the
stratification $(\{0,1\}, \, \mathbb{A}^1)$,  composition of the excision isomorphism and
the   forgetful map
give an isomorphism
\begin{equation}
\begin{tikzcd}[cramped, sep=small]
H^*(I, \,i_I^! \FF^\bullet)\arrow[r,equal]
& H^*_I(\PP,\, j_{\infty !} \FF^\bullet) \arrow[r]
& \Ho{*}(\PP,\, j_{\infty!}\FF^\bullet)
\end{tikzcd}
\end{equation}
%For a perverse sheaf  $\FF\in\Shs$ vanishing cycles $\Van{{1-z}}{\FF}$ is isomorphic to the cohomology with support $H^0_I(\Gm, \FF)$.
%For $i\ne 0$ cohomologies $H^i_I(\Gm, \FF)$ vanish.
\label{basic}
\end{prop}
\begin{proof}
Let 
$j_{\setminus I}\colon \PP\setminus I\to \PP$ be the open embedding. Consider the exact triangle
\begin{equation}
\begin{tikzcd}[cramped, sep=small]
{i_I}_*i_I^! \FF^\bullet \arrow[r, equal]
&{i_I}_*i_I^! j_{\infty!}\FF^\bullet \arrow[r]
& j_{\infty!}\FF^\bullet \arrow[r] 
&  j_{\setminus I *} {j_{\setminus I}}^*j_{\infty!} \FF^\bullet \arrow[r, "+1"] 
& {}
\end{tikzcd}
\label{retract}
\end{equation}
The pair of topological space $(\PP\setminus I,\, \PP\setminus \{I, \infty\}) $ is homeomorphic to the pair $(\mathbb{A}^1,\, \Gm)$.
The restriction of  $\FF^\bullet$  
on $\PP\setminus \{I, \infty\}$ is a complex of local systems.
To show that cohomology of its $!$-extension on the bigger space
vanishes, consider the corresponding complex of local systems on $\Gm$.
One may see that for a 
local system $\LL$ on  $\Gm$, cohomology  $\Ho{*}(\mathbb{A}^1,\, j_{\infty!} \LL)$ vanishes, which follows, that cohomology of a complex of local systems vanishes as well.
It follows that cohomology of the last term in (\ref{retract})
vanish.
Thus, cohomologies of first two terms in (\ref{retract})
are isomorphic.
\end{proof}

Let $j\colon \Gm \to \PP$ be the embedding of the multiplicative group to 
the projective line.
Factor $j$ as $j_\infty \circ j_0$, where $j_0$ is the
inclusion of $\Gm$ into $\mathbb{A}^1$, and $j_\infty$ is the inclusion of $\mathbb{A}^1$ into $\PP$.

Let $\FF$ be a perverse sheaf from $\Shs$.
Denote by $(0, 1]\subset \Gm$ the intersection of $I$ and $\Gm$.
Combining the canonical isomorphism of local cohomologies
 $$H^\bullet_{(0,1]}(\Gm,\, \FF) \simeq H^\bullet_{(0,1]}(\mathbb{A}^1,\, j_{0*}\FF) $$
with the excision isomorphism 
 $$ H^\bullet_{(0,1]}(\mathbb{A}^1,\, j_{0*}\FF)\simeq H^\bullet_{(0,1]}(\PP,\, \jj\FF) $$
 followed by 
 the  forgetful map   to $H^\bullet(\PP,\, \jj\FF)$ gives a map from
$H^0_{(0,1]}(\Gm,\, \FF)$ to  $H^0(\PP,\, \jj\FF)$.

\begin{prop}
For a  perverse sheaf $\FF\in \Shs$  
cohomology $H^0_{(0,1]}(\Gm,\, \FF)$ is naturally isomorphic to vanishing cycles $\Van{{1-z}}{\FF} $ of $\FF $ at $1$ 
and 
the  map defined above gives an isomorphism
\begin{equation}
\begin{tikzcd}[cramped, sep=small] \ph{I} \colon \Van{{1-z}}{\FF} \arrow[r] &  \Ho{0}(\PP,\, \jj \FF) \end{tikzcd}
\label{equality}
\end{equation} 
 For $i\ne 0$ cohomologies $\Ho{i}(\PP,\, \jj \FF)$ vanish.
\label{phi}
\end{prop}

\begin{proof}
Apply Proposition \ref{basic} to the perverse sheaf $j_{0!}\FF$.
As $H^0_{(0,1]}(\Gm,\, \FF)$ is equal to $H^0_I(\mathbb{A}^1,\, j_{0!}\FF)$, it gives  isomorphism (\ref{equality}). An isomorphism 
between $H^0_{(0,1]}(\Gm,\, \FF)$ and $\Van{{1-z}}{\FF}$
is established in \cite{Galligo1985} and may be taken as a definition of vanishing cycles. Moreover, it is shown there, that
 the exact triangle
\begin{equation}
\begin{tikzcd}[cramped]
H^\bullet(i_1^! \FF)\arrow[r]
& H^\bullet_{(0, 1]}(\Gm,\,  \FF) \arrow[r] 
& H^\bullet_{(0, 1)}(\Gm,\, \FF) \arrow[r, "+1"]
& {}
\end{tikzcd}
\label{complex}
\end{equation}
associated with the pair $(0, 1)\subset (0, 1]$ in $\Gm$   is isomorphic to the cohomology of the Milnor triangle (\ref{var}) for $\FF$, see also
 \cite[Remark 9.5]{KapSch}.
\end{proof}

Above we constructed not only the class in $\Ho{1}(\PP,\, \jj \FF[-1])$,
but the corresponding extension. This extension obeys the
natural  condition, which essentially specify in which direction 
interval $I$ meets point $1$, see Proposition \ref{frame} below.

\begin{rem}
Following \cite[Section 3]{Katz+1991}, one may show that the dimension of
$\Van{{1-z}}{\FF}$ coincides with the one of $\Ho{0}(\PP,\, \jj \FF)$ by means of 
the exactness properties of push-forwards along affine morphisms of perverse sheaves, see
\cite[4.1]{BBD}. Indeed, they imply that 
$\Ho{>0}(\PP,\, \jj \FF)=\Ho{>0}(\PP\setminus\{0\},\, j_!^\infty\FF)=0$
and also $\Ho{<0}(\PP,\, \jj \FF)=0$, by the Verdier duality. Thus, 
dimension of  $\Ho{0}(\PP,\, \allowbreak \jj \FF)$ equals to the Euler characteristic
of the cohomology, which can be shown to be equal to the dimension of $\Van{{1-z}}{\FF}$. This method does not suggest an explicit isomorphism between these spaces.
\label{purite} 
\end{rem}

\subsection{Nearby cycles}
For  a perverse sheaf on $\PP$,
which is an extension of 
a shifted local system from  $\PP\setminus \{0, \infty, 1\}$
by $*$, $*$, and $!$ in some order, Proposition \ref{phi} implies an isomorphism
between vanishing cycles at points, to which the local system is extended by $*$.
These vanishing cycles are isomorphic to nearby cycles,
the isomorphism is given by $var$ from (\ref{var}).
The following proposition reveals the nature of this isomorphism.

Denote by $j_1\colon \Gm\setminus\{1\} \to \Gm$ the open embedding
and by 
\begin{equation}
r\colon z\mapsto 1-z 
\label{reflexion}
\end{equation}
the automorphism of $\PP$.

\begin{prop}
Let $\LL$ be a local system on $\Gm\setminus \{1\}$.
The composition of isomorphisms
\begin{equation}
\begin{tikzcd}[cramped]
\Nea{{1-z}}{j_{1*}\LL[1]}
& \Van{{1-z}}{j_{1*}\LL[1]} \arrow[r, "\ph{I}"] \arrow[l, "var"']
& \Ho{0}(\PP,\, \mathop{j_{\infty!}j_{0*} j_{1*}}\LL[1])\arrow[d,"r"] \\
 \Nea{{z}}{j_{1*}\LL[1]}
& \Van{{z}}{j_{1*}\LL[1]}  \arrow[l, "var"'] \arrow[r, "\ph{I}"]
& \Ho{0}(\PP,\, \mathop{j_{\infty!}j_{1*} j_{0*}}\LL[1])
\end{tikzcd}
\end{equation}
is the holonomy along $I$.
\label{hol}
\end{prop}

\begin{proof}
The statement follows immediately from the proof of Proposition \ref{phi}:
for $j_{1*}\LL[1]$ the isomorphism $\phi_I$ identify with sections
of $i_I^! \LL[1]$ over the set of inner points of $I$, 
and the identification of it with nearby cycles consists in taking limits of these sections when approaching to the ends.
\end{proof}

\begin{rem}
Thus, $\Ho{0}(\PP,\, \mathop{j_{\infty!}j_{0*} j_{1*}}\LL[1])$ may be thought as the space  of sections of $\LL$ over the interval $(0,1)$ and
isomorphisms $\ph{I}$ as the specialization of the one in the nearby 
cycles at the ends, compare \cite[Introduction]{DGal}.
\end{rem}

Morphisms between different extensions of a smooth perverse sheaf 
on $\Gm\setminus \{1\}$ may be expressed in terms of Milnor triangles.

\begin{prop}
The following diagram commutes
\begin{equation}
\begin{tikzcd}[cramped]
\Ho{0}(i_1^!\FF) \arrow[r] \arrow[d, equal]
& \Ho{0}(\PP,\, \jj \FF) \arrow[r] 
& \Ho{0}(\PP,\, \mathop{j_{\infty!}j_{0*} j_{1*}j_1^*}\FF)\\
\Ho{0}(i_1^!\FF)  \arrow[r]
& \Van{{1-z}}{\FF} \arrow[r, "var"] \arrow[u,"\ph{I}"]
& \Nea{{1-z}}{\FF}=\Van{{1-z}}{j_{1*}j_1^*\FF} \arrow[u,"\ph{I}"] 
\end{tikzcd}
\end{equation}
where the second line is the cohomology of the Milnor triangle (\ref{var})
and the first line is the cohomology of the standard trianlge
\begin{equation}
\begin{tikzcd}[cramped]
{i_1}_*i_1^! \FF\arrow[r]
& \jj\FF \arrow[r] 
& \mathop{j_{1*}j_1^*j_{\infty!}j_{0*} }\FF \arrow[r, "+1"]
& {}
\end{tikzcd}
\end{equation}
\label{nea2}
\end{prop}

\begin{proof}
The statement follows immediately from the proof of Proposition \ref{phi}.
\end{proof}

The  statements analogous to two propositions above may be formulated for the $!$-extension as well.

An element $v\in \Van{{1-z}}{\FF} $ gives an extension, which represents the class
$\Ho{1}(\PP,\allowbreak\, \jj \FF[-1])$, given by Proposition \ref{phi}.
Combining this extension with the connecting morphism in the standard triangle
\begin{equation}
\begin{tikzcd}[cramped]
\mathop{i_{0*}i_0^! j_{0!}}\FF\arrow[r]
& \mathop{j_{0!} j_{\infty!} }\FF \arrow[r] 
& \mathop{j_{0*} j_{\infty!} }\FF \arrow[r, "+1"]
& {}
\end{tikzcd}
\label{zero}
\end{equation}
we get an extension representing $\Ho{1}(i_0^! j_{0!}\FF)$.

On the other hand, applying to $\var (v)$ isomorphism 
from  Propositions \ref{hol} we get an element of $\Nea{{z}}{j_{0!}\FF}$ at point $0$.
With such an element, one may associate an extension  representing $\Ho{1}(i_0^! j_{0!}\FF)$ by means of the Milnor triangle      (\ref{var}).

\begin{prop}
Two introduces above extensions representing class in 
$\Ho{1}(i_0^! j_{0!}\FF)$ coincide.
\label{frame}
\end{prop}

\begin{proof}
The statement follows immediately from the standard definition of nearby cycles
and the Milnor triangle.
\end{proof}

\begin{rem}
The Verdier duality gives isomorphisms between $\Ho{0}(\PP,\, \mathop{j^a_!j^b_* j^c_*}\LL[1])$ and $\Ho{0}(\PP,\, \mathop{j^a_*j^b_! j^c_!}\LL[1])$
for $\{a,b,c\}=\{0, \infty, 1\}$. Combining it with  natural maps
from $!$-extensions to $*$-extensions as in the proposition above,
one may get the action on the fiber of $\LL$ of monodromies at all infinite 
points, that is the action of the fundamental group of $\Gm\setminus\{1\}$.
\end{rem}

\section{Convolution}

\subsection{Moduli spaces}

Denote by $\M{n}$ the moduli space of  embeddings of an $n$-element set
$\set{n}\hookrightarrow \mathbb{P}^1$
to the complex projective line 
considered up to the
action of the M\"obius group.
This is a smooth affine variety, and it has a smooth
projective compactification $\Ms{n}$, which is the
moduli space of  stable curves.
The complement $\Ms{n}\setminus\M{n}$
is the union of normal crossing divisors.
These divisors are numerated by
partitions of $S$ in two subsets with cardinalities more than $2$.
Denote by $p_i\colon \M{n} \to \M{n-1}$ and $p_i\colon \Ms{n} \to \Ms{n-1}$ the
projection, forgetting the $i$-th point.

Introduce on $\PP$
 coordinates such that coordinates of first three elements of  $\set{n}$ are $0$, $\infty$ and $1$
correspondingly. 
Coordinates of  points labeled by the rest of elements of the finite set
are called simplicial coordinates on $\M{n}$.
Thus, a point of $\M{n+3}$ with simplicial coordinates $(t_1, \dots, t_n)$
is  $(0, \infty, 1, t_1, \dots t_n)$.

Simplicial coordinates identify $\M{4}$ with $\PP\setminus\{0, \infty, 1\}$,
and $\M{5}$ with $\M{4}\times\M{4}\setminus{\Delta}$, where $\Delta$
is the diagonal. The projections from $\M{5}$ to $\M{4}$ look
in simplicial coordinates as follows:
\begin{equation}
\begin{aligned}
p_3(0, \infty, 1, t_1, t_2)&=(0, \infty, 1, t_2/t_1)\\
p_4(0, \infty, 1, t_1, t_2)&=(0, \infty, 1, t_2)\\
p_5(0, \infty, 1, t_1, t_2)&=(0, \infty, 1, t_1)
\end{aligned}
\label{e5}
\end{equation}

Space $\Ms{4}$ is the obvious compactification of $\M{4}$, 
which is $\PP$. $\Ms{5}$ is isomorphic to 
$\Ms{4}\times\Ms{4}$ with blown-up points $(0,0)$, $(1,1)$ and $(\infty, \infty)$.

Denote by $\Ms{5}'$ the space $\Ms{4}\times\Ms{4}$ with blown-up points $(0,0)$ and $(\infty, \infty)$. In other words, $\Ms{5}'$ is $\Ms{5}$
with the blown-down line, corresponding to singular stable curves of 
type $(12)(345)$.
Denote by
\begin{equation}
p'_i \colon \Ms{5}' \to \Ms{4} 
\label{mp}
\end{equation}
the forgetful maps, as in (\ref{e5}).

\subsection{Convolution}

Let $m\colon \Gm\times\Gm\to \Gm$ be  the product map of the multiplicative group.
The convolution of two perverse sheaves $\EE$ and $\FF$ on $\Gm$ is
a perverse sheaf on $\Gm$
defined by
\begin{equation}
\con{\EE}{\FF} = \Rm m_!(\EE\boxtimes\FF)
\label{box}
\end{equation}
Note that because  map $m$ is affine, the convolution
of perverse sheaves is a perverse sheaf due to exactness property of affine morphisms, see  \cite[4.1]{BBD}.

\begin{rem}
Generally, $\operatorname{\mathbf{R}}m_!(\EE\boxtimes\FF)$ is a complex of perverse 
sheaves rather than a perverse sheaf, that is why we add $\tau^{\le 0}$
in the definition. One may avoid it, if, following \cite{DTerICM,DTerPr, DLet},
he works in the quotient category as in Theorem \ref{Quotient} because,
as it follows from the proof of Proposition \ref{conv2}, the higher terms
are smooth at $1$.
\end{rem}

Let $\imath\colon \PP \to \PP$
be the inversion map $z\mapsto z^{-1}$.
For a sheaf $\FF$ on $\PP$ introduce the notation
$$
\overline{\FF}= \imath^*\FF
$$

Recall, that in Section \ref{deligne} we denoted by
$\jj$ the $*$-extension at $0$ and $!$-extension at $\infty$ of a perverse sheaf from 
$\Gm$ to $\PP$. 

\begin{prop}[\cite{DLet, Katz+2012}]
For perverse sheaves $\EE$ and $\FF$ on $\Gm$
we have
\begin{equation}
\jj (\con{\EE}{\FF})=\Rm p'_{3*}({p'_5}^*(\overline{\jj\EE}) \otimes {p'_4}^*(\jj\FF))
\label{ext}
\end{equation}
where $p'_i$ are defined in (\ref{mp}).
\label{conv1}
\end{prop}

\begin{proof}
The proof proceeds in two steps. 

Firstly, 
one needs to build an isomorphism 
\begin{equation*}
\con{\EE}{\FF} \cong \Rm p'_{3*}({p'_5}^*(\overline{\jj\EE}) \otimes {p'_4}^*(\jj\FF))
\end{equation*}
on $\Gms$.
By (\ref{e5}), on fibers of the projection $p_3$, sheaves
 ${p'_5}^*(\overline{\jj\EE}) \otimes {p'_4}^*(\jj\FF))$ and  $\EE\boxtimes\FF$
are isomorphic  out of $0$ and $\infty$. Then one may see that the former 
is the $!$-extension the latter, it implies that its $*$-pushforward
equals to $!$-pushforward of the latter.

Secondly, one need to show that $\Rm p'_{3*}({p'_5}^*(\overline{\jj\EE}) \otimes {p'_4}^*(\jj\FF))$ is  $*$-extension to $0$
and $!$-extension to $\infty$ of $\con{\EE}{\FF}$. Consider the fiber
of $p'_3$ over $0$. It consists of two intersecting lines,
and ${p'_5}^*(\overline{\jj\EE}) \otimes {p'_4}^*(\jj\FF)$
is $*$-extension of $\EE\boxtimes\FF$ to these lines.
Thus, it is true for the pushforward as well. 
Over $\infty$ the situation is the same.
\end{proof}

\begin{prop}[\cite{DLet, Katz+2012}]
For perverse sheaves $\EE, \FF\in \Shs$, 
\begin{equation}
\Ho{0}(\PP,\, \jj(\con{\EE}{\FF}))=\Ho{0}(\PP,\, \jj\EE)\otimes H^0(\PP,\,\jj\FF)
\label{otimes}
\end{equation}
\label{conv2}
\end{prop}

\begin{proof}
Consider the Leray spectral sequence for 
$\R p'_{3*}({p'_5}^*(\overline{\jj\EE})\allowbreak \otimes {p'_4}^*(\jj\FF))$.
One may see that cohomologies of the derived pushforward
 outside $1$ is $*$-extension to $0$ and $!$-extension to $\infty$
of a smooth object. By Proposition \ref{phi}, it follows that
only the zeroth row of the spectral sequence does not vanish,
thus, it collapses at the second sheet. 
By the K\"unneth theorem, the cohomology of the derived pushforward
 equals to the cohomology of the right side of 
(\ref{otimes}). Because the second sheet of the Leray
spectral sequence consists  of one row, these cohomologies are presented
by 0th cohomology of the pushforward, thus, coincide with
the one for $\Rm p'_{3*}({p'_5}^*(\overline{\jj\EE}) \otimes {p'_4}^*(\jj\FF))$. 
\end{proof}

\subsection{Vanishing cycles of the convolution}
\label{category}

Proposition \ref{conv2} above and its proof shows that 
$H^0(\PP,\jj-)$ is a tensor functor on the category 
of perverse sheaves on $\Gm$ with respect to the tensor
structure given by convolution. 
By Proposition \ref{phi}, this functor is isomorphic to 
vanishing cycles at $1$.
This functor is not faithful, so it is not 
a fiber functor. It may be fixed by taking an appropriate quotient
of $\Shs$.

The following theorem is proved in \cite{DLet, DTerPr}. We will not need it and
 cite it for completeness.

\begin{theorem}[\cite{DTerICM,DTerPr, DLet}] Consider the subcategory of $\Shs$ formed by perverse sheaves
with nilpotent monodromy and  take the Serre quotient of it
by the subcategory of perverse sheaves smooth on $\Gm$.
This is a Tannakian category with the fiber functor
given by vanishing cycles at $1$ and the tensor product given 
by the multiplicative convolution.
\label{Quotient}
\end{theorem} 

\begin{rem}
The category $\mathrm{Perv}_0$ introduced in
\cite[12.6]{Katz+1991} is closely related with the above construction, 
it gives a splitting of the analogous quotient,
but for the category of perverse sheaves on the affine line with the
additive convolution as a tensor product.
\end{rem}

\section{Singular fibers}

\subsection{Sheaves without sections with support at 1}

%It is technically more convenient to deal with the following class of perverse sheaves.

We say that a perverse sheaf $\FF\in \Shs$  has no sections
with support at 1 if $H^0(i_1^! \FF)=0$. Equivalently,   sheaf  has no sections
with support at 1 if the map $\var\colon \Van{1-z}{\FF}\to \Nea{1-z}{\FF}$
induced by (\ref{var}) at point $1$ is injective. Two main examples are the $*$-extension and the middle extension of a smooth sheaf on $\Gm\setminus\{1\}$. 
An important feature of such sheaves is  that for them
(\ref{complex}) defines a two-term complex, representing
vanishing cycles.

\begin{prop}
If one of perverse sheaves $\EE, \FF\in\Shs$
 has no sections with support at $1$, then the convolution 
$\con{\EE}{\FF}$ has  no sections with support at $1$ as well.
\label{support}
\end{prop}

\begin{proof}
Calculate $i_1^!(\con{\EE}{\FF})$ by means of the base change for 
the map $p'_3$. Let $f\colon p^{\prime -1}_3(1)\to \Ms{5}'$
be the embedding of the fiber. We need to calculate
$$\Ho{0}(p^{\prime -1}_3(1), f^!({p'_5}^*(\overline{\jj\EE}) \otimes {p'_4}^*(\jj\FF))$$.

The fiber $p^{\prime -1}_3(1)$  is the projective line.
Denote by $s$ the embedding of point $(0, \infty, 1,1,1)$ in this line,
and by $t$ the embedding of the compliment. 
Denote ${p'_5}^*(\overline{\jj\EE}) \otimes {p'_4}^*(\jj\FF)$ by $\mathcal{A}$ and consider the standard triangle
\begin{equation*}
\begin{tikzcd}[cramped]
s_*s^!f^!{\mathcal{A}}\arrow[r]
& f^!{\mathcal{A}}\arrow[r] 
& t_*t^*f^!{\mathcal{A}}\arrow[r, "+1"]
& {}
\end{tikzcd}
\end{equation*}
Taking cohomology, we get
\begin{equation}
\begin{tikzcd}[cramped]
\Ho{0}(s^!f^!{\mathcal{A}})\arrow[r]
& \Ho{0}(\PP,\, f^!{\mathcal{A}})\arrow[r] 
& \Ho{0}(\PP,\, t_*t^*f^!{\mathcal{A}})\arrow[r, "+1"]
& {}
\end{tikzcd}
\label{hhh}
\end{equation}
where $\PP$ is $p^{\prime -1}_3(1)$. The first term is equal
to $\Ho{0}(i_1^!\EE)\otimes \Ho{0}(i_1^!\FF)$ by the K\"unneth theorem and 
vanishes by the condition of the proposition. One may see that
$t_*t^*f^!{\mathcal{A}}$ is isomorphic to 
$\mathop{j_{\infty!}j_{0*} j_{1*}j_1^*}(\overline{\EE}\otimes\FF)[-2]$,
where $j_1$ is $\Gm\setminus\{1\}\hookrightarrow\Gm$. 
One may see that $i$-th cohomology of the latter sheaf vanishes
for $i\ne 1$,
applying Proposition \ref{phi} to $\mathop{j_{1*}j_1^*}(\overline{\EE}\otimes\FF)[-1]$. It implies that the third term in (\ref{hhh}) vanishes,
thus the middles term vanishes too.
\end{proof}

\subsection{Nearby cycles of convolution at 0 via the Verdier specialization}

For a perverse sheaf from $\Shs$ without sections with support at $1$,
one may think about vanishing cycles at $1$ as a subspace of
nearby cycles at $1$. For the convolution 
%$\con{\EE}{\FF}$
 of
sheaves without sections
with support at 1,
one may  define a natural
embedding from the tensor product of vanishing cycles of factors
$\Van{{1-z}}{\EE}\otimes\Van{{1-z}}{\FF}$ 
to the nearby cycles at $0$ of the convolution $\Nea{{z}}{\con{\EE}{\FF}}$
in the following way.

Consider the fiber of $p_3'$ (or, equivalently, $p_3$) over $0$.
This is  the nodal stable curve, consisting of two components.
%corresponding to 
%singular stable curves
%of types $(15)(234)$ and $(135)(24)$. 
%Let us calculate the nearby cycles of ... on this fiber.
The convolution
is the shifted pushforward under $p_3'$ of the
sheaf ${p'_5}^*(\overline{\jj\EE}) \otimes {p'_4}^*(\jj\FF))$.
To calculate nearby cycles of the pushforward of ${p'_5}^*(\overline{\jj\EE}) \otimes {p'_4}^*(\jj\FF)$,
calculate the nearby cycles $\nea{z}{{p'_5}^*(\overline{\jj\EE}) \otimes {p'_4}^*(\jj\FF)}$ on the singular fiber, 
and then take the pushforward.

Let in
\begin{equation}
s\!:\, * \to p_3^{-1}( 0) \leftarrow \mathbb{A}^1 \cup \mathbb{A}^1 \, :\! t_{1,2}
\label{sing}
\end{equation}
$s$ be the embedding of the singular point, corresponding to the stable curve of type $(15)(3)(24)$, and $t_i$ be the embedding
of the $i$th complement.
The first component, corresponding to stables curves
of types $(15)(234)$ and the second one --- to $(135)(24)$. 

Introduce  coordinates on the compactification of the first and the second component as follows (see the left part of the picture at the end of the text):
\begin{equation*}
(0, \infty, 1)= ((15)(4)(23), (15)(3)(24) , (15)(2)(34))
\end{equation*}
and:
\begin{equation*}
(0, \infty, 1)= ( (15)(3)(24), (13)(5)(24),  (35)(1)(24))
\end{equation*}

Denote by $\EE_0$  the shifted vector bundle on $\Gm$ with fiber $\Nea{az}{\EE}[1]$ at point $a$ and   by $\FF_0$  the shifted vector bundle on $\Gm$ with fiber $\Nea{az}{\FF}[1]$ at point $a$. One may see that
$\EE_0$ and $\FF_0$ are isomorphic to restriction on $\Gm$ of Verdier specializations (\cite{Ver83})
of $\EE$ and $\FF$ along $0$ after identification of the tangent space to $\PP$
at $0$ with $\mathbb{A}^1=\PP\setminus\infty$ in the natural way.

Denote ${p'_5}^*(\overline{\jj\EE}) \otimes {p'_4}^*(\jj\FF)$ by $\mathcal{A}$.

One may see that on the first component in coordinates as above ${p'_5}^*(\overline{\jj\EE})$ is 
isomorphic to the restriction of $\overline{\jj\EE}$ and nearby cycles of 
$ {p'_4}^*(\jj\FF))$ along $z$ is isomorphic to $j_{\infty!}\FF_0$.
Since nearby cycles of tensor product is isomorphic to 
the tensor product of nearby cycles, we get
\begin{equation}
t_{1*}t^*_1\nea{z}{{p'_5}^*(\overline{\jj\EE}) \otimes {p'_4}^*(\jj\FF)}=
\overline{\jj(\EE\otimes \FF_0)}
\label{near1}
\end{equation}
and, analogously,
\begin{equation}
t_{2*}t^*_2\nea{z}{{p'_5}^*(\overline{\jj\EE}) \otimes {p'_4}^*(\jj\FF)}=
\jj(\EE_0\otimes \FF)
\label{near2}
\end{equation}
where we identify $\PP$, where right-hand sides of 
(\ref{near1}) and (\ref{near2}) are defined, with closures of our components.

Consider   the standard triangle associated with  stratification
(\ref{sing}):
\begin{equation*}
\begin{tikzcd}[cramped]
s_*s^!\nea{z}{\mathcal{A}}\arrow[r]
& \nea{z}{\mathcal{A}} \arrow[r] 
& t_*t^*\nea{z}{\mathcal{A}} \arrow[r, "+1"]
& {}
\end{tikzcd}
\end{equation*}
Thus, $\nea{z}{\mathcal{A}}$ may be presented as
\begin{equation}
\mathop{cone}(\,
\begin{tikzcd}[cramped, sep=small]
 t_*t^*\nea{z}{\mathcal{A}} \arrow[r]
& s_*s^!\nea{z}{\mathcal{A}}[1]\end{tikzcd}
\,),
\label{cone}
\end{equation}
where
\begin{equation}
 t_*t^*\nea{z}{\mathcal{A}}= t_{1*}t_1^*\nea{z}{\mathcal{A}}\,\oplus\, t_{2*}t_2^*\nea{z}{\mathcal{A}}
\label{sum}
\end{equation}

Recall that we want to build a map from $\Van{{1-z}}{\EE}\otimes\Van{{1-z}}{\FF}$ 
to  $\Nea{{z}}{\con{\EE}{\FF}}$. To do it, we will build 
extensions of $\underline{\mathbb{Z}}$ by 
$t_{1*}t_1^*\nea{z}{\mathcal{A}}$\and $t_{2*}t_2^*\nea{z}{\mathcal{A}}$
and ensure that their projections on $s_*s^!\nea{z}{\mathcal{A}}[1]$ in (\ref{cone})
coincide.

Combining (\ref{near1}) and (\ref{near2}) with Proposition \ref{phi},
we get
\begin{equation}
\begin{gathered}
\Ho{-1}(\PP,\, t_{1*}t^*_1\nea{z}{\mathcal{A}})=\Van{1-z}{\EE\otimes \FF_0}=\Van{1-z}{\EE}\otimes\Nea{z}{ \FF}\\
\Ho{-1}(\PP,\, t_{2*}t^*_2\nea{z}{\mathcal{A}}) =\Van{1-z}{\EE_0\otimes \FF}=\Nea{z}{\EE}\otimes\Van{1-z}{ \FF}
\end{gathered}
\label{gath}
\end{equation}
where the second equalities are due to the fact that
$\EE_0$ and $\FF_0$ are smooth at $1$ and the corresponding nearby
cycles are their shifted fibers there.

Given $u \otimes v \in \Van{{1-z}}{\EE}\otimes\Van{{1-z}}{\FF}$ ,
using isomorphisms (\ref{gath}) we get elements
\begin{equation}
\begin{gathered}
v\otimes (\Hol{I} (\var( u)))\in \Van{1-z}{\EE}\otimes\Nea{z}{ \FF}=\Ho{-1}(\PP,\, t_{1*}t^*_1\nea{z}{\mathcal{A}})\\
(\Hol{I} (\var( v)))\otimes u \in\Nea{z}{\EE}\otimes\Van{1-z}{ \FF}=\Ho{-1}(\PP,\, t_{2*}t^*_2\nea{z}{\mathcal{A}})
\end{gathered}
\label{gath2}
\end{equation}
where isomorphisms $\Hol{I}\colon \Nea{1-z}{-}\to \Nea{z}{-}$ are as in Proposition \ref{hol}.
Taking the difference between the element given by the first and second lines of
(\ref{gath2}) and taking into account (\ref{sum}) we get an element of
$\Ho{-1}(t_*t^*\nea{z}{\mathcal{A}})$.

To fulfill the construction, we need to check that projections on $s_*s^!\nea{z}{\mathcal{A}}[1]$ in (\ref{cone})
of extensions  given by the first and second lines of
(\ref{gath2}) coincide. To do it, on need to notice,
that these projections are given by the second arrow of the triangle (\ref{zero})
and apply Proposition \ref{frame}.

%Taking cohomology, we get
%\begin{equation*}
%\begin{tikzcd}[cramped, sep=small]
%0=\Ho{0}(i^!_s\nea{z}{\mathcal{A}})\arrow[r]
%& \Nea{z}{\con{\EE}{\FF}}=\Ho{1}(\nea{z}{\mathcal{A}}) \arrow[r] 
%& \Ho{1}(j^*_{t}\nea{z}{\mathcal{A}}) \arrow[r]
%& \Ho{1}(i^!_s\nea{z}{\mathcal{A}})
%\end{tikzcd}
%\end{equation*}
%Thus, the nearby cycles of the convolution is the kernel of the last arrow.
%Let us calculate last two terms and the map between them.
%
%
%Thus, we get
%\begin{equation*}
%\Ho{?}(j^*_{t}\nea{z}{\mathcal{A}})=
%\Ho{?}(\PP, \jj( \EE\otimes\FF_0))\oplus\Ho{?}(\PP,\, \jj( \EE_0\otimes\FF)),
%\end{equation*}
%
%
%By Proposition \ref{phi}, we have isomorphisms 
%\begin{gather*}
%zefz
%\end{gather*}

\begin{prop}
The map from $\Van{{1-z}}{\EE}\otimes\Van{{1-z}}{\FF}$
to $\Nea{{z}}{\con{\EE}{\FF}}$ defined above is the composition of maps
\begin{equation}
\begin{tikzcd}[scale=05, cramped, sep=small]
\Van{}{\EE} \otimes \Van{}{\FF}\arrow[r, "\ph{I}\otimes\ph{I}\!\!", "\sim"'] 
& \Ho{0}(\PP,\, \jj \FF)\otimes\Ho{0}(\PP,\, \jj \EE) \arrow[r,"\mbox{\tiny Prop.\ref{conv2}}",  "\sim"'] 
& \Ho{0}(\PP,\, \jj( \con{\EE}{\FF})) \arrow[d]\\
\Nea{{z}}{\con{\EE}{\FF}}
&\Van{{z}}{j_{1*}j_1^*(\con{\EE}{\FF})} \arrow[l, "\sim", "var_0"']
&\hspace{-.7cm}\Ho{0}(\PP,\, \mathop{j_{1*}j_1^*j_{\infty!}j_{0*} }(\con{\EE}{\FF}))
 \arrow[l, start anchor={[shift={(-17pt,0pt)}]}, 
      end anchor={[shift={(0pt,0pt)}]} , "\ph{I}^{-1}"', "\sim"]
\end{tikzcd}
\label{map1}
\end{equation}
where all arrows, except the vertical one, are isomorphisms, and the vertical 
arrow is the canonical map.
\label{Map1}
\end{prop}

\begin{proof}
The statement follows directly from definitions and is left to the reader.
\end{proof}

%For perverse sheaves, which are GM extension (?) of a smooth ones,
%the isomorphism \ref{map1} may be written in geometric terms.

\subsection{Nearby and vanishing cycles of convolution at 1  via the Verdier specialization}
\label{blowup}

Analyzing fibers of maps $p_3$ and $p_3'$ over $1$
one may identify vanishing cycles of convolution
with sections of the products of Verdier specializations of sheaves.
Below, we calculate vanishing cycles at $1$ and its image in nearby cycles of  convolution
of perverse sheaves from $\Shs$ without section with support  at $1$.

The fiber $p^{-1}_3(1)$ is  the nodal stable curve consisting of two components. 
Let in
\begin{equation}
s\!:\, * \to p_3^{-1}( 1) \leftarrow \mathbb{A}^1 \cup \mathbb{A}^1 \, :\! t\cup t'
\end{equation}
$s$ be the embedding of the singular point, corresponding to the stable curve of type $(12)(3)(45)$, $t$ be the embedding of the smooth part of the
component $(123)(45)$, and  $t'$ be the embedding of the smooth part of the component $(12)(345)$.

Introduce the coordinates on the compactification of the first and the second component as follows (see the right part of the picture at the end of the text):
\begin{equation*}
(0, \infty, 1)= ((13)(2)(45), (23)(1)(45) , (12)(3)(45))
\end{equation*}
and:
\begin{equation*}
(0, \infty, 1)= ( (12)(4)(35), (12)(3)(45),  (12)(5)(34))
\end{equation*}

The fiber $p^{\prime-1}_3(1)$ is the projective line. The natural projection 
$p^{-1}_3(1)\to p^{\prime-1}_3(1)$ blows down the component $(12)(345)$ and is an isomorphism on another. 
Keep notations $s$ and $t$ for the embeddings of  point $(12)(3)(45)$ and
its complement  in $p^{\prime-1}_3(1)$ :

\begin{equation}
s\!:\, * \to p_3^{\prime-1}( 1) \leftarrow \mathbb{A}^1  \, :\! t
\label{sing2}
\end{equation}

Denote $p_5^*(\overline{\jj\EE}) \otimes p_4^*(\jj\FF)$ also by $\mathcal{A}$.

As above, for sheaves from $\Shs$ without section with support at $1$,
we consider vanishing cycles at $1$ as a subspace of nearby cycles.
Thus, as it follows from the Milnor triangle
(\ref{var}), to build  an element from $\Van{1-z}{\con{\EE}{\FF}}$
one need to build an element from $\Nea{1-z}{\con{\EE}{\FF}}=\Ho{-1}(p_3^{\prime-1}( 1), \nea{1-z}{\mathcal{A}})$,
which vanishes under the connecting morphism of the cohomology
of the Milnor triangle to 
$\Ho{1}(i_1^! (\con{\EE}{\FF}))$.
We  claim that if $\EE$ or $\FF$ has no sections with support at $1$,
for this purpose one can take $v \in \Ho{-1}(p_3^{\prime-1}( 1), \nea{1-z}{\mathcal{A}})$ such that
\begin{equation}
\Ho{-1}(p_3^{\prime-1}( 1),\, t^*\nea{1-z}{\mathcal{A}}) \ni t^*v=0
\label{vanish}
\end{equation}
Indeed, let
$f\colon  p_3^{\prime-1}( 1) \to \Ms{5}'$ be the embedding  and
 consider the triangle associated with stratification (\ref{sing2})
for $f^!\mathcal{A}$:
\begin{equation*}
\begin{tikzcd}[cramped]
s_*s^!f^!{\mathcal{A}}\arrow[r]
& f^!{\mathcal{A}}\arrow[r] 
& t_*t^*f^!{\mathcal{A}}\arrow[r, "+1"]
& {}
\end{tikzcd}
\end{equation*}
Calculate its cohomology:
\begin{equation*}
\begin{tikzcd}[cramped ,sep=small]
\Ho{0}(s^!f^!{\mathcal{A}}) \arrow[r]\arrow[d, equal]
& \Ho{0}(p_3^{\prime-1}( 1), \,f^!{\mathcal{A}}) \arrow[r] 
& \Ho{0}(\mathbb{A}^1,\,t^*f^!{\mathcal{A}}) \arrow[d, equal]\\
\Ho{0}(i_1^!\EE)\otimes \Ho{0}(i_1^!\FF)
& {}
& \Ho{-1}(\mathbb{A}^1,t^*\nea{1-z}{\mathcal{A}})
\end{tikzcd}
\label{tri}
\end{equation*}
Here, the left vertical isomorphism is the K\"unneth theorem
and this term vanishes, because $\EE$ or $\FF$ has no sections with support at $1$. The right isomorphism follows from the fact that
${\mathcal{A}}$ is smooth in the neighborhood
of the image of $t$, hence  $\van{1-z}{\mathcal{A}}$ vanishes there.
Consider the image of $v$ under the connecting isomorphism in the middle term of
(\ref{tri}). Its image under the right arrow vanishes,
because it is equal to $t^*v$, which vanishes by (\ref{vanish}).
The left term is zero, which implies that the image is zero itself.

In order to build an element satisfying condition (\ref{vanish}),
express nearby cycles of $\con{\EE}{\FF}$ in terms of the special
fiber of the projection $p_3$ rather than $p_3'$.
Denote by $\overline{t}\colon \PP \to p_3^{-1}$ the closed embedding of the closure
of the image of $t$. Consider the stratification
\begin{equation}
\overline{t}\!:\, \PP \to p_3^{-1}( 1) \leftarrow \mathbb{A}^1 \, :\!  t'
\end{equation}
and the corresponding triangle for nearby cycles on $p_3^{-1}(1)$
\begin{equation}
\begin{tikzcd}[cramped]
t'_!t^{\prime *}\nea{1-z}{\mathcal{A}}\arrow[r]
& \nea{1-z}{\mathcal{A}} \arrow[r] 
& \overline{t}_*\overline{t}^*\nea{z}{\mathcal{A}} \arrow[r, "+1"]
& {}
\label{tt}
\end{tikzcd}
\end{equation}
For an element of $\Ho{-1}(\PP,\, t'_!t^{\prime *}\nea{1-z}{\mathcal{A}}) $, where $\PP$ denotes the closure of image of
$t'$, consider its image under the first arrow in (\ref{tt}).
Since cohomology of nearby cycles on $p_3^{-1}( 1)$ and $p_3^{\prime-1}( 1)$
coincide and give $\Nea{{1-z}}{\con{\EE}{\FF}}$, the image gives an element of 
$ \Ho{-1}(p_3^{\prime-1}( 1),\, \nea{1-z}{\mathcal{A}})$. We claim that this 
element satisfies condition (\ref{vanish}). Indeed, one may see that
$ t^*\nea{1-z}{\mathcal{A}}$ on both $p_3^{-1}( 1)$ and $p_3^{\prime-1}( 1)$
are canonically isomorphic. But $t^*$ on $p_3^{-1}( 1)$ factors through $\overline{t}^*$,
which annihilates the constructed element due to the exactness of (\ref{tt}).
Thus, an element from  $\Ho{-1}(\PP,\, t'_!t^{\prime *}\nea{1-z}{\mathcal{A}}) $
produces an element in nearby cycles of convolution, which may be lifted
to vanishing cycles.

\begin{prop}
For $\EE$ and $\FF$  from $\Shs$ without section with support at $1$, 
the defined above map to  vanishing cycles of $\con{\EE}{\FF}$
lifted from the map to nearby cycles
\begin{equation}
\begin{tikzcd}[cramped]
{}
& \Ho{-1}(\PP,\, t'_!t^{\prime *}\nea{1-z}{p_5^*(\overline{\jj\EE}) \otimes p_4^*(\jj\FF)})  \arrow[d]\arrow[ld, dashed]\\
\Van{{1-z}}{\con{\EE}{\FF}} \arrow[r, "var"]
&\Nea{{1-z}}{\con{\EE}{\FF}}
\end{tikzcd}
\end{equation}
is an embedding.
\label{Ver1}
\end{prop}

\begin{proof}
The statement follows directly from definitions and is left to the reader.
\end{proof}

Now calculate the domain of the map from the  proposition above.

For $\EE, \FF\in\Shs$ denote by
$\EE_1$ and $\FF_1$ the shifted vector bundles on
$\mathbb{A}^1$ which are Verdier specializations (\cite{Ver83})
of $\EE$ and $\FF$ along point $1$ after identification of the tangent space to $\PP$
at $1$ with $\mathbb{A}^1=\PP\setminus\infty$ in the natural way.

One may see that  in coordinates as above
$t^{\prime *}\nea{1-z}{{p'_5}^*(\overline{\jj\EE}}$ is isomorphic to $r^*\EE_1$ and 
$t^{\prime *}\nea{1-z}{ {p'_4}^*(\jj\FF))}$ is isomorphic to $\FF_1$,
where $r$ is defined in (\ref{reflexion}).
Thus, 
\begin{equation}
t^{\prime *}\nea{1-z}{\mathcal{A}}=r^*\EE_1\otimes\FF_1
\label{refl}
\end{equation}

\begin{prop}
For $\EE$ and $\FF$  from $\Shs$ without section with support at $1$, 
\begin{equation}
\Ho{0}(\PP,\, j^\infty_!(r^*\EE_1\otimes\FF_1)[-1])= \Van{{1-z}}{\EE}\otimes\Van{{1-z}}{\FF}
\label{Ver2f}
\end{equation}
\label{Ver2}
\end{prop}

\begin{proof}
Apply Proposition \ref{basic} to calculate the left-hand side of (\ref{Ver2f}).
The standard triangle for the local cohomology on $I$
associated with the stratification
$(\{0, 1\}, I)$
presents it as a complex in the derived category, that is a cone of
\begin{equation}
\begin{tikzcd}[cramped, sep=small]
\Nea{}{\EE_1}\otimes\Nea{}{\FF_1}\arrow[r]
& (\Ho{\bullet}(i_0^!\EE_1)[2]\otimes\Nea{}{\FF_1})\oplus (\Nea{}{\EE_1}\otimes\Ho{\bullet}(i_1^!\FF_1)[2])
\end{tikzcd}
\label{tens}
\end{equation}
Here, we use the identification of the cohomology 
the Milnor triangle (\ref{var}) with the complex given by stratification mentioned in the proof of Proposition \ref{phi}.
Analogously, applying  the canonical isomorphisms 
between nearby and vanishing cycles of perverse sheaves
and their Verdier specializations, one  may present 
the right-hand side of (\ref{Ver2f}) as a product of complexes
\begin{equation*}
\begin{tikzcd}[cramped, sep=small]
\Nea{}{\EE_1}\arrow[r]
& \Ho{\bullet}(i_0^!\EE_1)[2] 
\end{tikzcd} 
\quad \mbox{and} \quad
\begin{tikzcd}[cramped, sep=small]
\Nea{}{\FF_1}\arrow[r]
& \Ho{\bullet}(i_1^!\FF_1)[2]
\end{tikzcd}
\end{equation*}
One may see that there is a natural map from this tensor product
to (\ref{tens}), and the cone of this map is $\Ho{\bullet}(i_0^!\EE)[2]\otimes \Ho{\bullet}(i_1^!\FF)[2]$.
It follows that this map induces an isomorphism on the 0th cohomology,
because
$\EE$ and $\FF$ have no section with support at $1$.
\end{proof}

Combining Propositions \ref{Ver1} and \ref{Ver2} and (\ref{refl}),
we get for  perverse sheaves $\EE, \FF \in \Shs$ without section with support  at $1$ an isomorphism
\begin{equation}
\Van{{1-z}}{\EE}\otimes\Van{{1-z}}{\FF}=\Van{{1-z}}{\con{\EE}{\FF}} 
\label{map2}
\end{equation}

\begin{rem}
The left-hand side of (\ref{Ver2f}) is  the fiber at $1$ of the additive convolution of
Verdier specializations $\EE_1$ and $\FF_1$, which is isomorphic to
the nearby cycles of the additive convolution, which  in turn
is isomorphic to vanishing cycles under condition $\Ho{0}(i_0^!\EE_1)=\Ho{0}(i_0^!\FF_1)=0$, compare with the end of the first letter \cite{DLet}.
\end{rem}

\section{Vanishing cycles and polygons}
\label{poly}

\subsection{Graphical representation}

Isomorphism $\phi_I$ between vanishing cycles
$\Van{{1-z}}{\FF}$ and cohomology $\Ho{0}(\PP,\,\allowbreak \jj \FF)$
for $\FF\in \Shs$
supplied by Proposition \ref{phi} may be realized as
 a graphical representation of vanishing cycles in the following sense.
One associates a cocycle in 
$H^0(I,\, i_I^! \FF) with a vanishing cycle $ and then takes the pushforward to $\PP$ of it.
 Thus, one may think about interval $I$ and the cocycle
of the local cohomology with support on it as a graphical 
representation of the vanishing cycle.

Let us construct the graphical presentation of vanishing cycles of the convolution. To be more precise, we will construct its image
in  nearby cycles.

For two perverse sheaves $\EE, \FF\in \Shs$ take the external product
of the corresponding cocycles in $H^0_I(\PP,\,  -)$. It gives a cocycle with support on
the square in $\Gm\times\Gm$. The projection of this square under multiplication
is $I$. The preimage of this projection over an inner point of $I$
is the interval $[\lambda_1, \lambda_2]$ for 
$0<\lambda_1<\lambda_2<1$. The cocycle representing the cohomology with support on this interval, which we identify with $I$, is presented by the section of the restriction
of $\EE\boxtimes\FF$ on this interval, which is the product
of restrictions of factors. Note that restrictions
of this product at the ends of this interval vanish.

The picture at the end of the text could be helpful at this point,
 one should imagine that the right vertical side of
the pentagon-like gray figure is shrunk, 
making the pentagon in a quadrupole.

Thus, we get a  section of the local system $\Rde^1m_! (\EE[-1]\boxtimes\FF[-1])$ over inner points of $I$, 
which is isomorphic to  $\con{\EE}{\FF}$ there. After the pushforward,
 we get a class in $\Ho{0}(\PP,\, j_{1*}j_1^*(\con{\EE}{\FF}))$.
Combining this morphism with $\phi_I$, we get a class in
$\Van{{1-z}}{j_{1*}j_1^*(\con{\EE}{\FF})}=\Nea{{1-z}}{\con{\EE}{\FF}}$.

\begin{prop}
For $\EE$ and $\FF$  from $\Shs$ without section with support at $1$, 
 the defined above map 
\begin{equation}
\begin{tikzcd}[cramped, sep=small]
\Ho{0}(\PP,\, \jj\EE)\otimes \Ho{0}(\PP,\,\jj\FF)\arrow[r]
& \Nea{{1-z}}{\con{\EE}{\FF}}
\end{tikzcd}
\end{equation}
factors through
\begin{equation}
\begin{tikzcd}[cramped, sep=small]
\var\colon
\Van{{1-z}}{\con{\EE}{\FF}}\arrow[r]
& \Nea{{1-z}}{\con{\EE}{\FF}}
\end{tikzcd}
\end{equation}
and is an isomorphism on the image, which is
$\Van{{1-z}}{\con{\EE}{\FF}}$.
\label{square}
\end{prop}
\begin{proof}
By Proposition \ref{conv2}, the external product gives an isomorphism 
\begin{equation*}
\begin{tikzcd}[cramped, sep=small]
\Ho{0}(\PP,\, \jj\EE)\otimes H^0(\PP,\,\jj\FF)\arrow[r,"\sim"]
&\Ho{0}(\PP,\, \jj(\con{\EE}{\FF}))
\end{tikzcd}
\end{equation*}
By Proposition \ref{phi}, map $\phi_I$
establishes an isomorphism between  $\Ho{0}(\PP,\,\allowbreak \jj(\con{\EE}{\FF}))$  and 
$\Van{{1-z}}{\con{\EE}{\FF}}$.
Compose it with the composition
\begin{equation}
\begin{tikzcd}[cramped, sep=small]
\Van{{1-z}}{\con{\EE}{\FF}}\arrow[r]
&\Van{{1-z}}{j_{1*}j_1^*(\con{\EE}{\FF})}\arrow[r, equal]
&\Nea{{1-z}}{\con{\EE}{\FF}}
\end{tikzcd}
\label{comp}
\end{equation}
where the arrow is induced by the canonical map.
\begin{equation}
\begin{tikzcd}[cramped, sep=small]
\con{\EE}{\FF}\arrow[r]
&j_{1*}j_1^*(\con{\EE}{\FF})
\end{tikzcd}
\label{adj}
\end{equation}
By the very definition of $\phi_I$, this is 
the map constructed before the proposition.
By Proposition \ref{near2}, the composite map (\ref{comp}) is equal to $\var$.
By Proposition \ref{support}, 
$\con{\EE}{\FF}$ has no section with support at $1$, thus $\var$ is injective.
%This finishes the proof.
\end{proof}
%
%
%\begin{prop}
%
%\end{prop}
%\begin{proof}
%
%\end{proof}
%
%\begin{prop}
%
%\end{prop}
%\begin{proof}
%
%\end{proof}
%
%
%\subsection{Singular fibers}

\subsection{Compatibility}
Summarizing constructions from the previous section
with the graphical presentation introduced above, we get the following result.

\begin{theorem}
For $\EE, \FF \in \Shs$ 
without sections with support at $1$
the following diagram commutes
\begin{equation}
\begin{tikzcd}[cramped, column sep=1ex]
\Ho{0}(\PP, \, \jj j_{1*}j_1^* (\con{\EE}{\FF} ))\arrow[d, "\phi_I^{-1}\circ r^*"]
&\Ho{0}(\PP, \, \jj (\con{\EE}{\FF} ))\arrow[d, equal, "(\ref{otimes})" ]
\arrow[ddr, bend left, "\phi_I^{-1}", "\sim"'{sloped}, thin]\arrow[l, "(\ref{adj})"']
&{}\\
\Van{0}{j_{1*}j_1^*(\con{\EE}{\FF})}\arrow[d, "\var"]
&\Ho{0}(\PP,\, \jj\EE)\otimes \Ho{0}(\PP,\,\jj\FF)\arrow[d, equal, "\phi_I\otimes\phi_I"] 
&{}\\
\Nea{0}{\con{\EE}{\FF} }
&\Van{{1}}{\EE}\otimes\Van{{1}}{\FF}\arrow[l, "(\ref{map1})"'] \arrow[r, "(\ref{map2})", "\sim"']
&\Van{1}{\con{\EE}{\FF}}
\end{tikzcd}
\label{theorem}
\end{equation}
where  we denote by $\Van{{1}}{-}$, $\Van{{0}}{-}$ and $\Nea{{0}}{-}$
vanishing cycles at $1$ along $1-z$ and vanishing and nearby cycles at $0$
along $z$ correspondingly.
%and the left bent arrow is the composition of the morphism induced by the canonical map (\ref{adj}),  $\phi_I^{-1}$ for $r^*j_{1*}j_1^*(\con{\EE}{\FF})$ and the isomorphism given by $\var$ between
%vanishing and nearby cycles at $1$.
%
% isomorphism 
%given by Proposition \ref{phi} for $\jj(\con{\EE}{\FF})$,
%vanishing cycles of which we identify with  nearby cycles of $\con{\EE}{\FF}$ as in Proposition \ref{nea2}.
\label{Theorem}
\end{theorem}

\begin{proof}
Let us start with the right triangle of the diagram. The right isomorphism $\phi_I$ may be factored through 
\begin{equation*}
\begin{tikzcd}[cramped, sep=small]
\Van{{1-z}}{\con{\EE}{\FF}}\arrow[r]
&\Van{{1-z}}{j_{1*}j_1^*(\con{\EE}{\FF})}\arrow[r, equal]
&\Nea{{1-z}}{\con{\EE}{\FF}}
\end{tikzcd}
\end{equation*}
as in Proposition \ref{square}.
In that proposition, we present elements of
as cohomology with support on $I$, that is as a section of 
$\con{\EE}{\FF}$ over $I$. 
By the very definition of $\phi_I$ in Proposition \ref{phi}
via \cite{Galligo1985}, the
image of such a class in $\Nea{1}{\con{\EE}{\FF}}$
 is the limit of this section.
If we use the blown-up fiber over $1$ as in Subsection \ref{blowup},
on the picture at the end of the text, this section is presented 
by the right vertical side of the gray pentagon-like figure.
The limit of it is the bold line connecting $4$ and $5$ drawn on the figure presented,  the blown-up fiber over $1$. Thus, the class of nearby cycles
where $\phi_I$ landed is exactly the class, constructed in 
Proposition \ref{Ver2} by means of the shrinking arguments from Proposition \ref{basic}. 

The left part of the diagram may be treated likewise.
But this fact is essentially the content of Proposition \ref{Map1}.
\end{proof}

The theorem gives a map 
from vanishing cycles of the convolution at $1$ to its nearby cycles at $0$.
%On the other hand, by Proposition  \ref{hol}, these two spaces are connected
%by an isomorphism, which is the holonomy along $I$. 
The following proposition interprets
this maps in terms of  holonomy along $I$.

%\begin{prop}
%The map above is an isomorphism.
%
%\end{prop}
%\begin{proof}
%
%\end{proof}

%The  theorem above may be rephrased in  as follows. 

\begin{prop}
In notations of Theorem \ref{Theorem} the following diagram commutes
\begin{equation}
\begin{tikzcd}[cramped]
\Nea{0}{\con{\EE}{\FF} }\arrow[d, equal]
&\Van{{1}}{\EE}\otimes\Van{{1}}{\FF}\arrow[l, "(\ref{map1})"'] \arrow[r, "(\ref{map2})", "\sim"']
&\Van{1}{\con{\EE}{\FF}}\arrow[d, "\var"]\\
\Nea{0}{\con{\EE}{\FF} }\arrow[rr, equal, "\Hol{I}"]
&{}
&\Nea{1}{\con{\EE}{\FF} }
\end{tikzcd}
\end{equation}
where the top row is the bottom row of (\ref{theorem}) and the isomorphism downstairs
is the holonomy along $I$.
\end{prop}
\begin{proof}
By the very construction preceding Proposition \ref{square}, images of maps from $\Van{{1}}{\EE}\otimes\Van{{1}}{\FF}$
to nearby cycles  coincide with the ones of
$\Ho{0}(\PP,\, j_{1*}j_1^*(\con{\EE}{\FF}))$ under isomorphisms $\phi_I$.
Then, the statement  is a corollary of Proposition \ref{hol}.
\end{proof}

The following picture illustrates the proposition above.
The gray pentagon-like figure is the graphical representation
of the cohomology of the convolution. It may be thought of as well
as a trace of a graphical representation of a cohomology class in the fiber
when transporting along $I$. 

\clearpage
 
\medskip
\begin{figure}[h]
\includegraphics{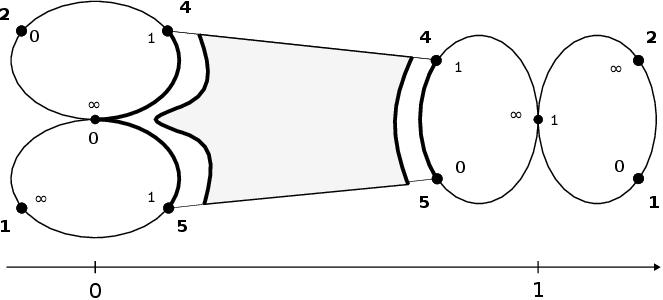}
\centering
\end{figure}

\bibliographystyle{alpha}
\bibliography{dt.bib}

\begin{thebibliography}{BBDG82}

\bibitem[BBDG82]{BBD}
Alexander Beilinson, Joseph Bernstein, Pierre Deligne, and Ofer Gabber.
\newblock Faisceaux pervers.
\newblock volume 100 of {\em Ast{\'e}risque}. Soc. Math. France, Paris, 1982.

\bibitem[Del89]{DGal}
Pierre Deligne.
\newblock Le groupe fondamental de la droite projective moins trois points.
\newblock In {\em Galois groups over Q. MSRI publications}, volume~16, pages
  72--297. Springer-Verlag, 1989.

\bibitem[Del01]{DLet}
Pierre Deligne.
\newblock Letters to {G}. {R}acinet.
\newblock April 2001.

\bibitem[EF21]{Enriquez2021}
Benjamin Enriquez and Hidekazu Furusho.
\newblock The {B}etti side of the double shuffle theory. {I}. {T}he harmonic
  coproducts.
\newblock {\em Selecta Mathematica}, 27(5):79, Aug 2021.

\bibitem[EF22]{Enriquez2022}
Benjamin Enriquez and Hidekazu Furusho.
\newblock The {B}etti side of the double shuffle theory. {II}. {D}ouble shuffle
  relations for associators.
\newblock {\em Selecta Mathematica}, 29(1):3, Oct 2022.

\bibitem[EF23]{Enriquez2023}
Benjamin Enriquez and Hidekazu Furusho.
\newblock The {B}etti side of the double shuffle theory. {III}. {B}itorsor
  structures.
\newblock {\em Selecta Mathematica}, 29(2):27, Mar 2023.

\bibitem[Fur11]{Furusho2011}
Hidekazu Furusho.
\newblock Double shuffle relation for associators.
\newblock {\em Ann. of Math.}, 174(1):341--360, 2011.

\bibitem[Fur22]{Furusho2022}
Hidekazu Furusho.
\newblock The pentagon equation and the confluence relations.
\newblock {\em American Journal of Mathematics}, 144(4):873--894, 2022.

\bibitem[GGM85]{Galligo1985}
André Galligo, Michel Granger, and Philippe Maisonobe.
\newblock {$\mathcal{D}$}-modules et faisceaux pervers dont le support
  singulier est un croisement normal.
\newblock {\em Annales de l'institut Fourier}, 35(1):1--48, 1985.

\bibitem[Hai21]{HainTur}
Richard Hain.
\newblock Hodge theory of the {T}uraev cobracket and the {K}ashiwara-{V}ergne
  problem.
\newblock {\em J. Eur. Math. Soc.}, 23(12):3889--3933, 2021.

\bibitem[HS19]{HiroseSato}
Minoru Hirose and Nobuo Sato.
\newblock Iterated integrals on {$P^1\setminus\{0,1,\infty,z\}$} and a class of
  relations among multiple zeta values.
\newblock {\em Advances in Mathematics}, 348:163--182, 2019.

\bibitem[IKZ06]{IKZ}
Kentaro Ihara, Masanobu Kaneko, and Don Zagier.
\newblock Derivation and double shuffle relations for multiple zeta values.
\newblock {\em Compositio Mathematica}, 142(2):307–338, 2006.

\bibitem[Kat91]{Katz+1991}
Nicholas~M. Katz.
\newblock {\em Exponential Sums and Differential Equations. (AM-124), Volume
  124}.
\newblock Princeton University Press, Princeton, 1991.

\bibitem[Kat12]{Katz+2012}
Nicholas~M. Katz.
\newblock {\em Convolution and Equidistribution}.
\newblock Princeton University Press, Princeton, 2012.

\bibitem[KS16]{KapSch}
Mikhail Kapranov and Vadim Schechtman.
\newblock Perverse sheaves over real hyperplane arrangements.
\newblock {\em Annals of Mathematics}, 183:619--679, 2016.

\bibitem[Mas01]{Massey2001}
David~B. Massey.
\newblock The {S}ebastiani–{T}hom {I}somorphism in the {D}erived {C}ategory.
\newblock {\em Compositio Mathematica}, 125(3):353–362, 2001.

\bibitem[Rac02]{Racinet}
Georges Racinet.
\newblock Doubles m\'elanges des polylogarithmes multiples aux racines de
  l'unit\'e.
\newblock {\em Publications Math\'ematiques de l'IH\'ES}, 95:185--231, 2002.

\bibitem[ST71]{Sebastiani1971}
Marcos Sebastiani and René Thom.
\newblock Un résultat sur la monodromie.
\newblock {\em Inventiones mathematicae}, 13:90--96, 1971.

\bibitem[TD05]{DTerPr}
Tomohide Terasoma and Pierre Deligne.
\newblock Harmonic shuffle relation for associators.
\newblock Unfinished prepint, 2005.

\bibitem[Ter06]{DTerICM}
Tomohide Terasoma.
\newblock Geometry of multiple zeta values.
\newblock In {\em Proceedings oh the International Congress of Mathematicians},
  volume~2, pages 627--636, 2006.

\bibitem[Ver83]{Ver83}
Jean-Louis Verdier.
\newblock Sp\'ecialisation de faisceaux et monodromie mod\'er\'ee.
\newblock In {\em Analyse et topologie sur les espaces singuliers (II-III) - 6
  - 10 juillet 1981}, number 101-102 in Ast\'erisque, pages 332--364.
  Soci\'et\'e math\'ematique de France, 1983.

\end{thebibliography}

\end{document}